\documentclass[a4paper,10pt]{article}

\usepackage[utf8]{inputenc}
\usepackage[english]{babel}
\usepackage[all,cmtip]{xy}
\usepackage{setspace}
\usepackage{multirow,rotating,graphicx}

\usepackage{amsmath,amsfonts,amssymb}
\usepackage{amsthm}
\usepackage{booktabs}
\usepackage{array}
\usepackage{latexsym,mathtools}

\usepackage{enumitem}
\usepackage{tikz,fancyvrb} 
\usetikzlibrary{arrows.meta,arrows,automata}
\usepackage{datetime}
\usepackage{caption}

\usepackage[colorlinks, citecolor=blue]{hyperref}
\usepackage[a4paper]{geometry}
\usepackage[big]{layaureo} 
\usepackage{frcursive}

\theoremstyle{plain}
\newtheorem{te}{Theorem}[section]
\newtheorem{theorem}[te]{Theorem}
\newtheorem{lemma}[te]{Lemma}
\newtheorem{corollary}[te]{Corollary}
\newtheorem{proposition}[te]{Proposition}
\theoremstyle{definition}
\newtheorem{definition}[te]{Definition}
\newtheorem{example}[te]{Example}
\theoremstyle{remark}
\newtheorem{remark}[te]{Remark}

\newcommand{\field}[1]{\mathbb{#1}}
\newcommand{\R}{\field{R}}
\newcommand{\N}{\field{N}}
\newcommand{\eps}{\varepsilon}
\newcommand{\lie}[1]{{\mathfrak{#1}}}
\newcommand{\g}{{\lie{g}}}
\newcommand{\Sl}{\lie{sl}}
\newcommand{\hook}{\lrcorner\,}
\newcommand{\st}{\;|\;}
\newcommand{\tran}[1]{\prescript{t}{}{#1}}
\DeclareMathOperator{\dd}{\textsl{d}}
\DeclareMathOperator{\tr}{tr}
\DeclareMathOperator{\Id}{Id}
\DeclareMathOperator{\id}{Id}
\DeclareMathOperator{\Span}{span}
\DeclareMathOperator{\Ric}{Ric}
\DeclareMathOperator{\ric}{ric}
\DeclareMathOperator{\ad}{ad}
\DeclareMathOperator{\Tr}{tr}
\DeclareMathOperator{\Der}{Der}

\numberwithin{equation}{section}

\allowdisplaybreaks

\newcolumntype{C}{>{$}c<{$}}
\newcolumntype{L}{>{$}l<{$}}
\newcolumntype{R}{>{$}r<{$}}

\title{New Special Einstein Pseudo-Riemannian Metrics on Solvable Lie Algebras}

\author{Federico A. Rossi}

\date{\today\ \currenttime}

\begin{document}
\maketitle

\begin{abstract}
We exhibit a concrete procedure to construct Einstein pseudo-K\"ahler and para-K\"ahler metrics on solvable Lie algebras. We apply this method to classify all the rank-one pseudo-Iwasawa extensions of type-\ref{cond:nil4} nilsoliton in low dimension. We prove that such metrics exists on the rank-one pseudo-Iwasawa extension of the generalized Heisenberg Lie algebra.

Further ideas and suggestions to produce more special Einstein pseudo-Riemannian metrics are exposed.
\end{abstract}

\setcounter{tocdepth}{2}

\tableofcontents

\renewcommand{\thefootnote}{\fnsymbol{footnote}}
\footnotetext{\emph{MSC class 2020}: \emph{Primary} 53C50; \emph{Secondary} 53C25, 53C30, 22E25, 32M10}
\renewcommand{\thefootnote}{\arabic{footnote}}

\section{Introduction}
One of the most classical and studied problem in differential geometry deals with the existence and construction of \emph{Einstein metrics} $g$ on a given differentiable manifold $M$, that is metrics $g$ that satisfy the Einstein equation
\begin{equation}\label{eq:Einstein}
 \Ric_g=\lambda\Id\qquad \text{for some } \lambda\in\R.
\end{equation}
Those metrics are special solutions to the more general Einstein field equation $\ric_g+\lambda g= \kappa T$ where $T$ is the \emph{stress-energy tensor}, $\lambda\in\R$ is the \emph{cosmological constant} and $\kappa\in\R$ is the \emph{Einstein gravitational constant} (see e.g.~\cite{Weinberg:Gravitation}).
If $\lambda=0$ the metric is called \emph{Ricci-flat}. Ricci-flat metrics are related to string theory and they are fixed point of the Ricci-flow.

The special case of Riemannian homogeneous manifold is a very active field of research and many properties are known.
Solvmanifolds are particular homogeneous manifolds: those are simply connected solvable Lie group with a left-invariant metric. The Alekseevsky conjecture, stating that all Riemannian homogeneous Einstein manifolds of negative scalar curvature are solvmanifolds, is one of the most interesting problem in this setting, and very recently B\"{o}hm and Lafuente in the preprint~\cite{BohmLafuente:AlekseevskyConjecture} show a proof of it. Einstein Riemannian solvmanifold are constructed using nilsoliton metrics and the theory is well understood thanks to the works of many authors (see e.g. \cite{Lauret:Einstein_solvmanifolds,Nikolayevsky,Heber:noncompact}).

On the other hand, a K\"ahler metric represents another interesting object in geometry. A K\"ahler structure on a Lie algebras $\g$ is a triple $(J,\omega,g)$ consisting of a complex structure $J$, a symplectic structure $\omega$, a compatible Riemannian metric $g$ (i.e. $g(JX,Y)=\omega(X,Y)$) and such that $J,\omega,g$ are parallel: in some sense the K\"ahler geometry sits between the complex, the symplectic and the metric geometry. Special metrics are the Einstein K\"ahler ones, that combine the properties of both the previous classes. In small dimension one can find some examples in \cite{FerFinMan:G2EinSolv,Manero:phdThesis} (see also Remark~\ref{rem:ManeroMisprint}).

In this article we address the construction of invariant ``special'' Einstein metrics with nonzero scalar curvature on solvmanifolds in the pseudo-Riemannian setting: in particular we will deal with \emph{Einstein pseudo-K\"ahler metrics} and \emph{Einstein para-K\"ahler metrics}.

The pseudo-K\"{a}hler condition is the analogous notion of K\"ahler condition in the more general pseudo-Riemannian setting. Those special metrics can be found quite often in literature, and for Lie algebras some classification are obtained expecially in low dimension
(e.g. Ovando in~\cite{Ovando06:pseudoKahlerDim4} classifies pseudo-K\"ahler Lie algebras up to dimension $4$, while in~\cite{CorFerUga04:PseudoK6Nil} those metrics are studied in the nilpotent context).

Einstein para-K\"ahler metrics incorporate para-K\"ahler structure, where the complex structure $J$ is replaced by a \emph{para-complex} structure $K$ (i.e. an endomorphism of $TM$ such that $K^2=+\id$). Those special metrics were studied in~\cite{AleMedTom09:HomParaKEinstein} and~\cite{BenBou15:PKandHyperPK}. It is worth remembering that pseudo-K\"ahler (resp. para-K\"ahler) structures are the point of contact of symplectic, metric and complex (resp. para-complex) geometry, but where the metric is a more general pseudo-Riemannian metric.

Although Einstein condition on the one hand and ``special'' (pseudo-K\"ahler or para-K\"ahler) structures on the other have been studied extensively but separately in the literature, less has been done to study these two conditions simultaneously. This article aims to go in this direction, showing a concrete strategy for building examples of special pseudo-Riemannian Einstein metrics. In fact, our recent papers \cite{ContiRossi:IndefiniteNilsolitons,ContiRossi:NiceNilsolitons}, jointly with D.~Conti, have allowed us to construct new examples of Einstein homogeneous manifold with indefinite metrics. The new technique developed may be combined with other tools to address different problems about the construction of special metrics, as we aim to show in the present paper.

More precisely, we will deal with Einstein metrics that are not Ricci-flat, and we will show how to construct Einstein pseudo-K\"ahler metrics and Einstein para-K\"ahler metrics on solvable Lie algebra $\g$. Those structures can be easily transferred to invariant analogous structures on a connected, simply-connected Lie group $G$ whose Lie algebra is $\g$; and conversely, a left-invariant Einstein pseudo-K\"ahler (or para-K\"ahler) metric on a solvable Lie group $G$ can be viewed as a special (linear) structure on the corresponding Lie algebra $\g$.

\medskip
The paper is organized as follows. In Section~\ref{sec:preliminaries} we recall some useful preliminary notions. In the next Section~\ref{sec:SpecialMetrics}, we present the special metrics that we want to investigate, and in Section~\ref{section:NilpotentEinstein} we expose the problem of Einstein metrics, pseudo-K\"ahler and para-K\"ahler metrics on nilpotent Lie group.

Section~\ref{sec:PseudNilsolitonSummary} contains a brief summary of the results contained in~\cite{ContiRossi:IndefiniteNilsolitons}, and it explains how to obtain Einstein metrics with nonzero scalar curvature on solvable Lie algebras starting from pseudo-Riemannian nilsoliton metrics.

In Section~\ref{sec:strategy} we present our concrete strategy to build Einstein pseudo-K\"ahler (para-K\"ahler) metrics with nonzero scalar curvature. We also apply this strategy to obtain new examples, and we conclude this section with a classification of rank-one pseudo-Iwasawa extension of nice diagonal nilsoliton (Theorem~\ref{thm:NiceDiagonalPseudoIwasawaLowDim}).

Finally, in the last Section~\ref{sec:examples}, we introduce new different ideas that originate from the strategy explained in Section~\ref{sec:strategy}. Those ideas produce other new examples (Examples~\ref{ex:DoubleExtension421:1} and Example~\ref{ex:521:2ParaStorta}) and generalizations (Example~\ref{ex:GeneralizedHeisenberg} and Theorem~\ref{thm:GenerHeisemberg}). Thus this list may plot new routes to many other examples or to new results, and may help the comprehension of Einstein pseduo-Rimannian homogeneous manifolds.

\bigskip
\noindent \textbf{Acknowledgments:} This paper is an expanded version of a talk given by the author to the meeting ``Cohomology of Complex Manifolds and Special Structures -- II'' (Levico, TN, Italy, 5-9 July 2021).
The author would like to thank the organizers Prof. C.~Medori, Prof. M.~Pontecorvo and Prof. A.~Tomassini.
The author wishes to express his gratitude to Prof. D.~Conti for useful discussions, suggestions and clear comments on the topics presented in this paper.
The author acknowledge GNSAGA of INdAM and the Young Talents Award of Universit\`{a} degli Studi di Milano-Bicocca joint with Accademia Nazionale dei Lincei.

\section{Invariant Structures on Lie Groups}\label{sec:preliminaries}

We recall some definitions that we will use in the sequel. Given a Lie group $G$, we will denote by $\g$ its Lie algebra.

We will consider \emph{left-invariant pseudo-Riemannian metrics} $g$ on a given Lie group $G$; these will be expressed as (in)definite scalar products $g$ (i.e. bilinear nondegenerate symmetric forms) on the corresponding Lie algebra $\g$, and the pair $(\g,g)$ will be called a \emph{metric Lie algebra}. It follows that also the Levi-Civita connection $\nabla$, the Riemann curvature $R$, the Ricci operator $\Ric$ and the Ricci curvature $\ric$ are left-invariant, so they can be expressed as linear tensors on $\g$. If the metric is positive definite, it is called \emph{Riemannian metric}, otherwise it is said to be an \emph{indefinite metric}. Any left-invariant endomorphism $E$ of $G$ can be expressed as a linear endomorphism on $\g$, again denoted by $E$. A left-invariant \emph{symplectic structure} on a Lie group $G$ is a left-invariant closed $2$-form $\omega$ of maximal rank, and it is given by a nondegenerate closed $2$-form $\omega$ on the Lie algebra $\g$.

The \emph{lower central series} of a Lie algebra $\g$ is recursively defined as $\g^0:=\g$, $\g^i:=[\g,\g^{i-1}]$, and a Lie algebra is called \emph{nilpotent} if the lower central series becomes the trivial subspace $\{0\}$, i.e. if there exists a $s\in\N$ such that $\g^s=\{0\}$, in this case the minimum $s$ such that $\g^s=\{0\}$ is called \emph{step}.
The \emph{derived series} of $\g$ is defined as $\mathfrak{a}_0:=\g$, $\mathfrak{a}_i:=[\mathfrak{a}_{i-1},\mathfrak{a}_{i-1}]$. If there exist a $r\in\N$ such that $\mathfrak{a}_r=\{0\}$ then the Lie algebra is said to be \emph{solvable}, and the minimum integer $r$ is called \emph{deep}. The dimension of a complementary subspace of $\g'=[\g,\g]=\lie{a}_1$ is called \emph{rank}, so a \emph{rank-one solvable Lie algebra} as a vector space can be written as $\g=\g'\oplus \Span{X}$ for some $X\in\g$. A Lie algebra is \emph{unimodular} if $\tr\ad (X)=0$ for any $X\in\g$.
Obviously a nilpotent Lie algebra is also solvable and unimodular.

A metric is \emph{Einstein} if $\Ric= \lambda g$ for some constant $\lambda\in\R$ (equiv. $\ric = \lambda \Id$). The metric is called \emph{Ricci-flat} if $\lambda= 0$. The \emph{scalar curvature} will be denoted by $s:=\tr(\Ric)$, so Einstein metric with nonzero scalar curvature satisfies $s=n\lambda\neq0$.

An invariant \emph{almost complex structure} $J$ is an endomorphism of $\g$ such that $J^2=-\Id$, and an invariant endomorphism $K$ of $\g$ such that $K^2=+\Id$ whose eigenspaces have the same dimensions is called \emph{almost para-complex structure}. The \emph{torsion tensor} (or \emph{Nijenhuis tensor}) of an endomorphism $E$ is defined as
\[N_E(X,Y):=[EX, EY ]- E[EX, Y ]- E[X, EY ]+ E^2[X,Y].\]
If the Nijenhuis tensor of an almost complex structure $J$ (resp. of an almost para-complex structure $K$) vanishes, then the almost complex structure (resp. the almost para-complex structure) is called \emph{integrable} or simply a \emph{complex structure} (resp. a \emph{para-complex structure}).

We recall that an almost para-complex structure $K$ is equivalent to a decomposition of the Lie algebra $\g$ in two vector subspaces $\g_{+},\g_{-}$ of the same dimension $n$, (namely the eigenspaces of $K$), and $K$ is integrable if, and only if, the eigenspaces are involutive (i.e. $\g_{+},\g_{-}$ are subalgebras of $\g$).

A \emph{solvmanifold} $(G,g)$ is a simply connected solvable Lie group $G$ with a left-invariant metric $g$. A \emph{nilmanifold} $\Gamma\backslash G$ is a compact quotient of a nilpotent Lie group $G$ by a cocompact discrete subgroup $\Gamma$. By Mal\textquotesingle cev criterion \cite{Malcev}, a nilpotent Lie group has a cocompact discrete subgroup if and only if its Lie algebra $\g$ admits a basis with rational structure constants.

In the sequel of this paper, we will deal with left-invariant objects on Lie group $G$, so we will identify them with their linear counterparts on the Lie algebra $\g$. We also want to study Einstein metrics with nonzero scalar curvature. We remark that tensors on $\lie{g}$ define left-invariant tensors on the connected simply-connected lie group $G$ with Lie algebra $\g$, hence all our following discussion can be translated into Lie group language.

\section{Pseudo-Riemannian Einstein Special Metrics}\label{sec:SpecialMetrics}

Given an (almost) complex structure $J$ on a Lie algebra $\g$, a pseudo-Riemannian metric $g$ is a \emph{(almost) pseudo-Hermitian metric} if
\[g(JX,JY)=g(X,Y),\quad X,Y\in\g.\]
In this case we define {\emph{fundamental $2$-form} as \[\omega(X,Y)=g(JX,Y),\quad X,Y\in\g.\]
The Hermitian metric $g$ is \emph{(almost) pseudo-K\"{a}hler} if the fundamental $2$-form is closed, i.e. $\dd\omega=0$.

It follows directly from the definition that pseudo-K\"{a}hler geometry is complex, symplectic and the pseudo-Riemannian metric is compatible with both those structures. Moreover in the integrable case we get that $\nabla J=\nabla g=\nabla \omega=0$.

Similarily, a compatibility condition between a para-complex structure $K$ and a pseudo-Riemannian metric $g$ can be given. Namely a pseudo-Riemannian metric $g$ is an \emph{(almost) para-Hermitian} metric if
\[g(KX,KY)=-g(X,Y),\quad X,Y\in\g.\]
The \emph{fundamental $2$-form} $\omega$ is the nondegenerate form defined by
\[\omega(X,Y)=g(KX,Y),\quad X,Y\in\g\]
and an (almost) para-Hermitian metric is \emph{(almost) para-K\"{a}hler}} if $\dd \omega=0$. From the definitions, we see that any (almost) para-Hermitian $g$ has neutral signature $(n,n)$, and the eigenspaces of $K$ are null spaces for $g$.

Note that para-K\"{a}hler geometry is para-complex, symplectic and the metric is compatible with both those structures. Furthermore, easy computations show that in this case $K, g$ and $\omega$ are all parallel with respect to $\nabla$, the Levi-Civita connection of $g$.

If we also require the metric $g$ to be Einstein, then we deal with \emph{Einstein (almost) pseudo-K\"{a}hler} and \emph{Einstein (almost) para-K\"{a}hler} metrics.

\section{Einstein pseudo-Riemannian nilpotent Lie algebras}\label{section:NilpotentEinstein}

Our goal is to construct Einstein pseudo-K\"{a}hler (or para-K\"{a}hler) metrics on solvable Lie algebras with nonzero scalar curvature. To this end, we start by looking at the nilpotent case.

We begin recalling some important results of Milnor and Dotti Miatello about the (almost)-K\"ahler Einstein on solvable Lie algebra in the Riemannian case:

\begin{theorem}[{\cite{Milnor:curvatures}}]
Let $\g$ be a nilpotent not Abelian Lie algebra, then it has no Einstein Riemannian invariant metric.
\end{theorem}
By Milnor's result the only nilpotent Lie algebra admitting an Einstein metric is the abelian Lie algebra, and in this case the metric is Ricci-flat. This is a more general fact concerning unimodular solvable Lie algebra, as stated by the
\begin{theorem}[{\cite{Dotti:RicciCurvature}}]
Let $\g$ be a unimodular solvable Lie algebra, then any Einstein Riemannian metric $g$ on $\lie{g}$ is flat, hence Ricci-flat.
\end{theorem}

On the other hand, in the indefinite case there are examples of Einstein metric with $s\neq0$ on non-abelian nilpotent Lie algebras: the first example appeared in~\cite{ContiRossi:EinsteinNilpotent} and a systematic way to construct them was reported in \cite{ContiRossi:EinsteinNice,ContiRossi:RicciFlat}.

The construction of nilpotent Lie algebras admitting Einstein metrics with nonzero scalar curvature is obstructed by an algebraic condition. Let us recall that the space of derivations of a Lie algebra $\g$ is the set
\[\Der(\g):=\{X\colon \g\to \g\st X \text{ is linear and } X[v,w]=[Xv,w]+[v,Xw] \}.\]
Then we proved the following obstruction results:
\begin{theorem}[{\cite[Theorem~4.1]{ContiRossi:EinsteinNilpotent}}]
Let $\g$ be an unimodular Lie algebra with Killing form zero. If $\g$ has an Einstein metric with $s\neq 0$, then $\Der(\g)\subset \Sl(\g)$.
\end{theorem}

Moreover, for low dimensions we have that most of the real Lie algebras does not admit such metrics, because of the
\begin{theorem}[{\cite[Theorem~4.3 and Theorem~4.4]{ContiRossi:EinsteinNilpotent}}]\label{thm:ObstuctionLowDimEinstein}
On a nilpotent Lie algebra of dimension up to six, Einstein metrics are Ricci-flat.

If $\g$ is a nilpotent $7$-dimensional Lie algebra not appearing in
following Table~\ref{Table1}, every Einstein metric on $\g$ is Ricci-flat.
\end{theorem}

\begin{table}[thp]
{\small\centering
\begin{tabular}{RL}
\toprule
\textnormal{Name \cite{Gong}} & \g \\
\midrule
123457E & 0,0,e^{12},e^{13},e^{14},e^{23}+e^{15},e^{23}+e^{24}+e^{16} \\
123457H & 0,0,e^{12},e^{13},e^{14}+e^{23},e^{15}+e^{24},e^{25}+e^{23}+e^{16} \\
123457H_1 & 0,0,e^{12},e^{13},e^{14}+e^{23},e^{15}+e^{24},-e^{16}-e^{25}+e^{23} \\
13457I &  0,0,e^{12},e^{13},e^{14},e^{23},e^{25}+e^{26}-e^{34}+e^{15} \\
12457J &  0,0,e^{12},e^{13},e^{23},e^{24}+e^{15},e^{34}+e^{25}+e^{16}+e^{14} \\
12457J_1 &  0,0,e^{12},e^{13},e^{23},e^{24}+e^{15},e^{34}-e^{25}+e^{16}+e^{14} \\
 12457N & 0,0,e^{12},e^{13},e^{23},e^{24}+e^{15}, \lambda e^{25}+e^{26}+e^{34}-e^{35}+e^{16}+e^{14},\ \lambda\in\R \\
 12457N_1 &  0,0,e^{12},e^{13},e^{23},-e^{25}-e^{14},-e^{35}+e^{25}+e^{16} \\
 12457N_2 & 0,0,e^{12},e^{13},e^{23},-e^{14}-e^{25},e^{15}-e^{35}+e^{16}+e^{24}+ \lambda e^{25},\ \lambda\geq0 \\
 123457F & 0,0,e^{12},e^{13},e^{14},e^{15}+e^{23},e^{16}-e^{34}+e^{24}+e^{25} \\
12457G & 0,0,e^{12},e^{13},0,e^{25}+e^{14}+e^{23}, -e^{34}+e^{26}+e^{15} \\
\bottomrule
\end{tabular}
\caption{\label{Table1}$11$ cases with $\Der(\g)\subset\Sl(\g)$, i.e. that might carry an Einstein metric with $s\neq 0$.}
}
\end{table}

We know that Fern\'{a}ndez, Freibert and S\'{a}nchez found in~\cite{FerFreSan20:7dimEinsteinNilpotent} an Einstein metric on the Lie algebra $123457E$. We believe that also $123457H$, $123457H_1$, $12457J$, $12457J_1$ and $12457G$ have an invariant Einstein metric with $s\neq0$ (here the names refer to the Gong's classification given in~\cite{Gong}).

As said before, we were able in~\cite{ContiRossi:EinsteinNilpotent} to construct the first $8$-dimensional example of a pseudo-Riemannian Einstein metric with $s\neq0$ on a nilpotent Lie algebra. For the sake of completeness, we will report that example.

\begin{example}[{\cite[Theorem~5.2]{ContiRossi:EinsteinNilpotent}}]\label{ex:PrimoEinstein8Dim}
Consider the nilpotent Lie algerba $\g$ whose structure equations satisfy the following
\[(0,0,0,0,e^{12}+e^{34},e^{14}-e^{23},e^{16}-e^{24}+e^{35},-e^{13}+e^{26}+e^{45}).\]
Here and throughout the paper, we will describe Lie algebras by giving the action of the Chevalley-Eilenberg operator $\dd$ on the dual (which is equivalent to giving the expression for the Lie bracket). The notation above means that $\g^*$ has a fixed basis $\{e^1,\dotsc, e^4\}$ with $de^1=\dots =d e^4=0$, $de^5=e^{12}+e^{34}=e^1\wedge e^2+e^{3}\wedge e^4$ and so on.

The Lie algebra $\g$ admits two Einstein metric with scalar curvature $s=\frac{56}{15}\neq 0$, with signature $(6,2)$ and $(3,5)$, namely:
\[
e^1\otimes e^1 + e^2\otimes e^2\pm(e^3\otimes e^3+ e^4\otimes e^4)
-\frac73e^5\otimes e^5 \mp\frac73 e^6\otimes e^6
\pm \frac{98}{15}(e^7\otimes e^7 +  e^8\otimes e^8 ).
\]
\end{example}

The Lie algebra of Example~\ref{ex:PrimoEinstein8Dim} is a particular Lie algebra, called nice. A \emph{nice Lie algebra} is a pair $(\g,\mathcal{B})$ where $\g$ is a Lie algebra, $\mathcal{B}=\{e_1,\dots,e_n\}$ is a basis with structure constants $c_{ij}^k$ such that:
\begin{enumerate}
\item for all $i<j$ there exists at most one index $k$ s.t. $c_{ij}^k\neq0$;
\item the condition $c_{ij}^k, c_{lm}^k\neq0$ implies that either $\{i,j\}=\{l,m\}$ or $\{i,j\}\cap\{l,m\}=\emptyset$;
\end{enumerate}
(i.e. both $[e_i,e_j]$ and $e_i\hook\dd e^j$ are multiple of at most one element of the basis $\{e_i\}$ or of the dual basis $\{e^i\}$ up to constants). Two nice nilpotent Lie algebras $(\g,\mathcal{B})$, $(\g',\mathcal{B}')$ are considered \emph{equivalent} if there is a Lie algebra isomorphism $f$ that maps basis elements to multiples of basis elements.

To a nice Lie algebra $(\g,\mathcal{B})$ we can associate:
\begin{enumerate}
\item a directed graph $\Delta$ with arrows labelled by nodes, called its \emph{nice diagram}: the nodes of $\Delta$ are the elements of the nice basis, and $e_i\xrightarrow{e_j}e_k$ is an arrow if $[e_i,e_j]$ is a nonzero multiple of $e_k$;
\item the \emph{root matrix} $M_\Delta$, which has a row for every  nonzero $c_{ij}^k$; the row associated to $c_{ij}^k$ has $+1$ in position $k$, $-1$ in positions $i$ and $j$, and zeroes in the other entries.
\end{enumerate}

\begin{figure}[thp]
\begin{center}
\begin{tikzpicture}[scale=0.65]
\begin{scope}[every node/.style={circle,thick,draw}, minimum size=3em]
    \node (A) at (-6,0) {$1$};
    \node (B) at (-2,0) {$2$};
    \node (C) at (2,0) {$3$};
    \node (D) at (6,0) {$4$};
    \node (E) at (-4,-3) {$5$};
    \node (F) at (4,-3) {$6$};
\end{scope}
\begin{scope}[>={stealth[black]}, every edge/.style={draw,very thick}]
  \path [->] (A) edge              node [above]{$3$} (E);
  \path [->] (A) edge [out=240, in=240, looseness=1] node [above]{$2$} (F);
  \path [->] (B) edge              node [above]{$4$} (E);
  \path [->] (B) edge [left, looseness=1.5, in=90] node [above]{$1$} (F);
  \path [->] (C) edge              node [above]{$1$} (E);
  \path [->] (C) edge              node [above]{$4$} (F);
  \path [->] (D) edge [out=300, in=240, looseness=1.5] node [below]{$2$} (E);
    \path [->] (D) edge              node [above]{$3$} (F);
    \end{scope}
\end{tikzpicture}
\caption{A nice diagram $\Delta$.\label{fig:NiceDiagram}}
\end{center}
\end{figure}
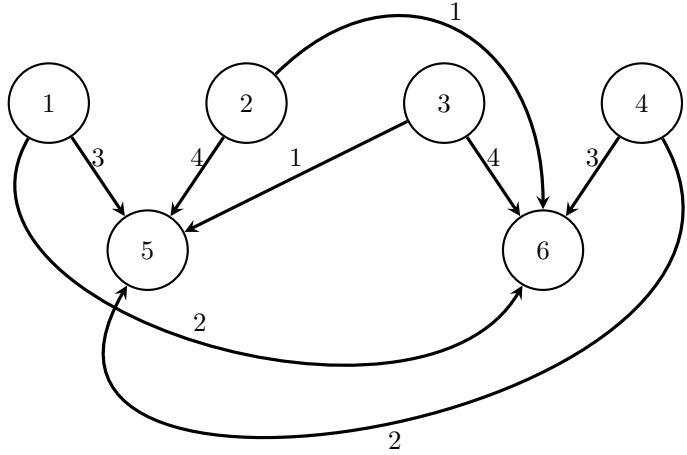

\begin{example}
Consider the $6$-dimensional nice Lie algebra given by \[(0,0,0,0,e^{13}+e^{24},e^{12}+e^{34}),\]
then its nice diagram $\Delta$ is drawn in Figure~\ref{fig:NiceDiagram}, and the corresponding root matrix is
\[M_{\Delta}=
\begin{pmatrix}
-1& 0 & -1& 0 & 1 & 0 \\
0 & -1& 0 & -1& 1 & 0 \\
-1& -1& 0 & 0 & 0 & 1 \\
0 & 0 & -1& -1& 0 & 1
\end{pmatrix}.
\]
\end{example}

We studied intensively the nilpotent nice Lie algebras, and we obtain a classification up to equivalence for dimension $\leq 9$ (\cite{ContiRossi:Construction}). More recently we develop helpful techniques to handle with them in~\cite{ContiDelBarcoRossi:NiceAdInvariant}. Those nice Lie algebras were introduced and studied in \cite{LauretWill:diagonalization,LauretWill:EinsteinSolvmanifolds}, and are an useful tool in the study of nilsoliton and Ricci flow in the Riemannian and pseudo-Riemannian setting (see e.g. \cite{Nikolayevsky,Payne:TheExistence,Lauret:Finding,Lauret:RicciSoliton,Payne:Applications,Will:RankOne,ContiRossi:NiceNilsolitons}), and lately they were used to address the problem of ad-invariant metrics (\cite{ContiDelBarcoRossi:Uniqueness,ContiDelBarcoRossi:NiceAdInvariant}).
We also used nice Lie algebra to construct explicit left-invariant Einstein pseudo-Riemannian metrics on nilpotent Lie group with $s\neq0$ (\cite{ContiRossi:EinsteinNice}) or $s=0$ (\cite{ContiRossi:RicciFlat,ContiDelBarcoRossi:DiagramInvolutions}); even if those examples are difficult to construct, we know that for each $n\geq8$, there exist $n$-dimensional nice nilpotent Lie algebras with an Einstein diagonal metric with $s\neq0$ \cite[Theorem~3.7]{ContiRossi:EinsteinNice}.

However, the latter construction seems quite useless to our goal because of the following results:
\begin{lemma}[{\cite[Lemma~6.3]{FPS:skt6d}}]\label{lem:FPSInvariantPseudoNilpotent}
If $(J, g)$ is an invariant pseudo-K\"{a}hler structure on a nilpotent Lie algebra $\lie{g}$, then the Ricci tensor of $g$ vanishes.
\end{lemma}

\begin{proposition}[{\cite[Corollary~8.2, Proposition~6.4]{ContiRossi:RicciTensor}}]\label{pr:ContiRossiInvariantParaK}
Invariant para-K\"{a}hler structures $(K,g)$ on nilpotent Lie algebras are Ricci-flat.

A left-invariant para-K\"ahler structure on a unimodular Lie group has $s=0$.
\end{proposition}

So, even if it is possible to build Einstein non-Ricci-flat metrics on nilpotent Lie algebras, it is impossible to merge this condition with further special structures, such as pseudo-K\"{a}hler or para-K\"ahler.

This force us to look for another category out of nilpotent condition to obtain special Einstein metrics. Reasonable targets are solvable Lie algebras, and we show below that such an example exists.

\begin{example}\label{ex:AffPK}
Consider the solvable Lie algebra $\lie{aff}(\R)$ of dimension $2$ defined by $(e^{12},0)$ and
consider the tensors:
\begin{gather*}
g_+= e^1\otimes e^1+ e^2\otimes e^2,\qquad g_-= -e^1\otimes e^1+ e^2\otimes e^2,\qquad\omega=e^{12},\\
J(e_1)=e_2,\quad J(e_2)=-e_1,\qquad K(e_1)=e_2,\quad K(e_2)=e_1.
\end{gather*}
A straightforward computation shows that $K^2=\id$, $J^2=-\id$, $N_K=N_J=0$, $\dd\omega=0$, $g_+(JX,Y)=\omega(X,Y)=g_-(KX,Y)$ and $\Ric_{g_+}=\Ric_{g_-}=-\Id$. So $(g_+,J,\omega)$ is an Einstein pseudo-K\"ahler metric, while $(g_-,K,\omega)$ is an Einstein para-K\"ahler metric, and both metrics have nonzero scalar curvature.
\end{example}

\section{Pseudo-Riemannian Nilsolitons and Einstein Solvmanifolds}\label{sec:PseudNilsolitonSummary}

By the results mentioned in the previous section, to have Einstein special metrics with nonzero scalar curvature we should look to a bigger class than the nilpotent one, namely the solvable Lie algebras. The first step is to study how to construct pseudo-Riemannian Einstein metrics on solvable Lie algebra.

In the indefinite case, there are new phenomenon and new instances that do not appear in the positive definite case: there are Einstein metrics on nilpotent Lie algebras with nonzero scalar curvature; there are example of Einstein metrics on solvable Lie algebra whose restriction to the nilradical is a degenerate bilinear form; there are abundance of Ricci-flat metrics, both on nilpotent and solvable non-nilpotent Lie algebra. In the end, we observe that Einstein pseudoriemannian metrics may apper either on unimodular or nonunimodular Lie algebras.

These differences and the wide spectrum of examples that can be constructed make a complete classification of Einstein metrics in the pseudo-Riemannian case very difficult. However, in a previous paper together with D.~Conti, we study the Einstein solvable Lie algebras and their relations with pseudo-Riemannian nilsoliton metrics~(\cite{ContiRossi:IndefiniteNilsolitons}).
In this section we recall some important facts and we will show how the study of nilsoliton metrics help the construction of Einstein metrics with nonzero scalar curvature. The interested reader may refer to \cite{ContiRossi:IndefiniteNilsolitons,ContiRossi:NiceNilsolitons}.

Our aim was to find some suitable categories $\mathcal{S}$ and $\mathcal{N}$, where
\begin{center}
$\mathcal{S}\subset\{$  Einstein pseudo-Riemannian solvmanifolds $\}$\\
$\mathcal{N}\subset\{$  pseudo-Riemannian nilsolitons $\}$
\end{center}
such that we can have a bijective relation:
\begin{center}
\begin{tabular}{ccc}
$\mathcal{S}\ni$ Einstein Solvmanifold $(\tilde{\g},\tilde{g})$  & $\longleftrightarrow$ &  Nilsoliton metric $(\g,g)$ $\in \mathcal{N}$.
\end{tabular}
\end{center}

This relation in the Riemannian setting is rigid and robust, by the work of many authors (see e.g. the surveys~\cite{Lauret:SurveyNilsoltions,Jablonski:Survey} and the reference therein), for example Lauret (\cite{Lauret:Einstein_solvmanifolds}) proved that Einstein solvmanifolds are standard, i.e. they decompose as an orthogonal direct sum
\[\tilde{\g}= \g \oplus^\perp \lie{a},\]
where $\g$ is a nilpotent ideal and $\lie{a}$ is an abelian subalgebra. Previous results of Heber (\cite{Heber:noncompact}) showed that $\lie{a}$ acts by normal derivations on $\g$ and we know that those examples are nonunimodular (\cite{Dotti:RicciCurvature}). Unfortunately, as reported before, we proved that most of the properties of the Riemannian case simply does not hold for the more general pseudo-Riemannian setting: nevertheless, we were able to obtain a cute relation that include the Riemannian one as a very special case.

The starting point to construct our yearned relation is the following:
\begin{definition}
A \emph{standard decomposition} of a metric Lie algebra $\tilde{\lie{g}}$ is a decomposition $\tilde{\lie{g}}=\lie{g}\oplus^\perp\lie{a}$, where $\lie{g}$ is a nilpotent ideal and $\lie{a}$ is an abelian subalgebra.
\end{definition}
Note that $\g$ is any nilpotent ideal of $\tilde{\g}$ which is settled between the derived Lie algebra $[\tilde{\lie{g}},\tilde{\lie{g}}]=\tilde{\lie{g}}'$ and the nilradical $\lie{n}$ of $\tilde{g}$, i.e. $\tilde{\lie{g}}'\subset\lie{g}\subset\lie{n}$. Note also that the standard Riemannian solvmanifolds are also standard for the pseudo-Riemannian definition.

In fact, in \cite[Section~1]{ContiRossi:IndefiniteNilsolitons}, we show that imposing any further restriction on the ideal $\g$ may lead to some problem on the nondegenerate of the metric: we found examples of Einstein solvable Lie algebra $\tilde{\g}$ where $(\tilde{\lie{g}}')^{\perp}$ is degenerate, and we found examples of Einstein solvable Lie algebra where $(\lie{n})^{\perp}$ is degenerate (so chosing $\g$ equals to the nilradical or to the derived Lie algebra necessarily excludes interesting cases).

On the other hand, this definition of standard decomposition is the minimum to get some interesting relations, since we need the metric to be nondegenerate on some nilpotent ideal. Note that some examples of Einstein solvable Lie algebra do not fit in this definition: in particular we are  excluding cases where $\lie{g}$ has degenerate metric and cases with $\lie{g}^\perp$ not abelian. Examples of either those instance may occur.

On a solvable Lie algebra with a standard decomposition, we define the vector $H$ to be the metric dual of $X\mapsto\Tr \ad X$, i.e. $\tilde{g}(H,X)=\Tr\ad X$, $X\in\tilde{\lie{g}}$.

Denoting by $E^*$ the adjoint of $E$ with respect to the metric, we introduce the following definition, that will play an important role in the sequel.
\begin{definition}
A standard decomposition $\tilde{\lie{g}}=\lie{g}\oplus^\perp\lie{a}$ is \emph{pseudo-Iwasawa} if $\ad X= (\ad X)^*$ for any $X\in\lie{a}$.
\end{definition}

Pseudo-Iwasawa condition seems a strong restriction, but we proved that other Einstein solvable Lie algebras are isometric to a pseudo-Iwasawa one (\cite[Proposition~1.19]{ContiRossi:IndefiniteNilsolitons}). Note also that a Riemannian Einstein solvable Lie algebra is always (pseudo-)Iwasawa (see~\cite{Heber:noncompact}). On the other hand, it is possible to have Einstein solvable pseudo-Riemannian Lie algebra where $(\ad X)^*$ is not a derivations, or examples of Einstein solvable Lie algebra such that $[\ad X,(\ad X)^*]\neq0$.

The second ingredient to obtain a relation is given by the notion of nilsoliton: the pair $(\lie{g},g)$ is a \emph{(Algebraic) nilsoliton} if
$\lie{g}$ nilpotent Lie algebra, $g$ a pseudo-Riemannian metric and the Ricci tensor satisfies
\begin{equation}\label{eq:nilsoliton}
\Ric_g=\lambda\id+D,\quad \lambda\in\R,\ D\in \Der\lie{g}.
\end{equation}

Thus Einstein metrics are particular solutions of~\eqref{eq:nilsoliton} with $D=0$. We remark that all the algebraic pseudo-Riemannian nilsolitons are Ricci soliton (\cite{Onda:ExampleAgebraicRicciSolitons}), but there exist Ricci solitons which are not algebraic~(\cite{BatatOnda:AlgebraicRicciSoliton}).

Recall that a derivation $N$ on a Lie algebra $\g$ is called a \emph{Nikolayevsky} (or \emph{pre-Einstein}) \emph{derivation} if it is diagonalizable and satisfies $\Tr(NX)=\Tr X$ for any $X\in\Der\g$. Note that nice nilpotent Lie algebra have Nikolayevsky derivation diagonal with respect to the nice basis.

We gave the following characterization of the nilsoliton metrics.

\begin{theorem}[{\cite[Theorem~2.1]{ContiRossi:IndefiniteNilsolitons}}]
\label{thm:nilsolitonsandnik}
Let $g$ be a nilsoliton metric on a nilpotent Lie algebra $\g$. Then either
\begin{enumerate}
\item $\lambda=0$ and $D$ is a nilpotent derivation in the null space of $\Der\g$; or
\item $\lambda\neq0$ and setting $\tilde D=-\frac1\lambda D$, we have
\[\Tr(X)=\Tr(\tilde D\circ X),  \quad X\in\Der\g;\]
relative to the Jordan decomposition $\tilde D=\tilde D_s+\tilde D_n$, $\tilde D_s$ is a Nikolayevsky derivation and $\tilde D_n$ a nilpotent derivation in the null space of $\Der\g$.
\end{enumerate}
In either case, the eigenvalues of $D$ are rational.
\end{theorem}

From this theorem we have that equation~\eqref{eq:nilsoliton} generates $4$ different type of nilsolitons in the pseudo-Riemannian case, that we can summarize as follow:
\begin{enumerate}[label={$(\mathrm{Nil\arabic*})$}]
\item\label{cond:nil1} $\lambda=0$, $D=0$. This is the Ricci-flat case, examples of which exist in abundance in the pseudo-Riemannian setting (see e.g. \cite{ContiDelBarcoRossi:DiagramInvolutions,ContiRossi:RicciFlat}).
\item\label{cond:nil2} $\lambda=0$, $D\neq0$. In this case, by Theorem~\ref{thm:nilsolitonsandnik} $D$ is nilpotent (nonsemisimple) and it is not a Nikolayevsky derivation.
\item\label{cond:nil3} $\lambda\neq0$, $D=0$. This is the Einstein case, with nonzero scalar curvature, that we discussed in Section~\ref{section:NilpotentEinstein}. This case has no Riemannian counterpart by Milnor~\cite{Milnor:curvatures}.
\item\label{cond:nil4} $\lambda\neq0$, $D\neq0$. This is the classical Riemannian situation. Note that if $D$ is diagonalizable, it is a multiple of a Nikolayevsky derivation. However, examples of indefinite nilsoliton with nonsemisimple derivation $D$ may appear.
\end{enumerate}

In our paper, we gave examples of nilsoliton metrics for each type (as well as the first example of type~\ref{cond:nil4} with nonsemisimple derivation), but more examples can be found in literature: for example, there are known examples of~\ref{cond:nil4} on the Heisenberg Lie algebra where the derivation $D$ is semisimple~\cite{Onda:ExampleAgebraicRicciSolitons,KondoTamaru:LorentzianNilpotent}. It is worth to mention (see \cite[Remark~2.6]{ContiRossi:IndefiniteNilsolitons}) that pseudo-Riemannian nilsolitons are not unique, unlike the Riemannian ones, as proved in~\cite{Lauret:RicciSoliton}.

On the other hand, we demonstrated that an Einstein pseudo-Iwasawa solvable Lie algebra induce a nilsoliton equation on the nilpotent ideal $\g$.
\begin{theorem}[{\cite[Theorem~3.9]{ContiRossi:IndefiniteNilsolitons}}]
\label{thm:solvablenilsolitoncorrespondence}
Let $\tilde\g=\g\oplus^{\perp}\lie{a}$ be a pseudo-Iwasawa decomposition. Then the Einstein equation $\Ric_{\tilde{g}} = \lambda \id$ on $(\tilde\g,\tilde{g})$ is satisfied if and only if
\begin{enumerate}
\item $\g$ with the induced metric satisfies the nilsoliton equation
\[\Ric = \lambda \id+D, \qquad D=\ad H.\]
\item $\langle \ad X, \ad Y\rangle_{\Tr}=-\lambda\tilde{g}(X,Y)$ for all $X,Y\in\lie a$.
\end{enumerate}
In this case, then
\[\Tr D^2 = -\lambda\Tr D.\]
\end{theorem}
Where $\langle X,Y\rangle_{\Tr}$ denotes the inner product induced by the trace, i.e. $\langle X,Y\rangle_{\Tr}=\Tr (X\circ Y)$.

The relation we were looking for is given by the following Corollary~\ref{cor:Hzero} and Theorem~\ref{thm:constructionstandardextension}.
\begin{corollary}[{\cite[Corollary~3.13]{ContiRossi:IndefiniteNilsolitons}}]
\label{cor:Hzero}
Given a pseudo-Iwasawa solvable Lie algebra $\tilde\g=\g\rtimes\lie a$ satisfying $\Ric_{\tilde{g}}=\lambda \id$ for some nonzero $\lambda$, then either:
\begin{enumerate}
\item  $\tilde\g$ is unimodular, $H=0$ and $(\g,g)$ is a nilsoliton of type~\ref{cond:nil3}, with $\Ric_g=\lambda \id$; or
\item $\tilde\g$ is not unimodular, $\tilde{g}(H,H)\neq0$, $\g\oplus\Span{H}$ is also Einstein with a pseudo-Iwasawa decomposition, and $(\g,g)$ is a nilsoliton of type~\ref{cond:nil4}, with $\Ric_g=\lambda\id+D$ and $\Tr D\neq0$.
\end{enumerate}
\end{corollary}

In particular, we will use the following theorem to obtain Einstein solvable Lie algebra from nilsoliton metrics (compare with \cite[Theorem~4.7]{Yan:Pseudo-RiemannianEinsteinhomogeneous}).

\begin{theorem}[{\cite[Theorem~4.1]{ContiRossi:IndefiniteNilsolitons}}]\label{thm:constructionstandardextension}
Let $(\g,g)$ be a nilsoliton, $\Ric_g=\lambda \id+D$, $\lambda\neq0$. Let $\lie a\subset\Der\g$ be a subalgebra of self-adjoint derivations containing $D$, and assume that $\langle,\rangle_{\Tr}$ is nondegenerate on ${\lie a}$. Then
the metric $g-\frac1{\lambda}\langle,\rangle_{\Tr}$ on $\tilde \g=\g\rtimes\lie a$ is Einstein with $\widetilde\Ric=\lambda \id$ and the decomposition $\tilde\g=\g\oplus^\perp\lie{a}$ is pseudo-Iwasawa.
\end{theorem}

We will call the solvable Lie algebra constructed in Theorem~\ref{thm:constructionstandardextension} a \emph{pseudo-Iwasawa} extension of the nilsoliton $\g$ and the dimension of $\lie{a}$ will be called the \emph{rank} of the extension. The problem to construct Einstein metrics on Lie algebra can be translate in computing nilsoliton metrics. However, in general this problem may be as difficult as to solve the Einstein equation~\eqref{eq:Einstein}.

Luckily, it turns out that on nice nilpotent Lie algebras it is easy to perform computations and to find diagonal nilsoliton metrics, as we will see from the next result, but before stating it, we need to fix some notation. For a vector $X$ we denote with $X^D$ the diagonal matrix with diagonal matrix with $X$ in the diagonal;
$[1]$ denotes a vector with all entries equal to $1$; $e^{M_\Delta}(g)$ is a vector depending on the root matrix: explicitly if the $h$-th row of $M_\Delta$ is $\tran(-e_i-e_j+e_k)$, the $h$-th component of
$e^{M_\Delta}(g)$ is $\frac{g_k}{g_ig_j}$ (where $g_i$ are the component of the diagonal metric); and $c$ is the vector of the structure constants $c_{ij}^k$ written using the same order as in the root matrix.

\begin{theorem}[{\cite[Theorem~1.5]{ContiRossi:NiceNilsolitons}}]
\label{thm:Xequalsb}
Let $\g$ be a nice nilpotent Lie algebra and let $b$ be a solution to $M_\Delta\tran{M_\Delta} b =[1]$. Given a diagonal pseudo-Riemannian metric $g$, the following are equivalent:
\begin{enumerate}
\item $g$ is a diagonal nilsoliton with $\Ric=\lambda \id+D$;
\item $\Ric = \lambda(\id-N)$, where $N$ is the diagonal Nikolayevsky derivation;
\item $X\in -2\lambda b + \ker\tran{M_\Delta}$, where $X^D= e^{M_\Delta}(g)(c^D)^2.$
\end{enumerate}
\end{theorem}

The last condition of the theorem above can be disassembled to linear problems and a polynomial one, and we applied this theorem obtaining the signature of all possible diagonal nilsoliton pseudo-Riemannian metric in low dimensions (see \cite[Theorem~2.1]{ContiRossi:NiceNilsolitons}).

The diagonal nilsoliton metric on nice Lie algebra force the derivation $D$ apperaing in~\eqref{eq:nilsoliton} to be diagonalizable, hence to be the diagonal Nikolayevsky derivation. Applying Theorem~\ref{thm:constructionstandardextension} we obtain the following
\begin{proposition}[{\cite[Proposition~1.20]{ContiRossi:NiceNilsolitons}}]
\label{prop:NicePseudoIwasawa}
Let $\g$ be a nice nilpotent Lie algebra with a diagonal metric $g$ of type~\ref{cond:nil4}. Then $\Ric_g=\lambda \id-\lambda N$, where $N$ is the (nonzero) diagonal Nikolayevsky derivation, and the semidirect product $\tilde\g=\g\rtimes_N\Span{e_0}$ is a nice solvable Lie algebra with an Einstein diagonal pseudo-Iwasawa metric
\[g-\frac{\Tr N}{\lambda}e^0\otimes e^0.\]
\end{proposition}

\begin{remark}\label{rem:RankOneNiceExtension}
Since a nice nilpotent Lie algebra often has many nilsoliton metrics with different signatures, we see that in the case of diagonal nice nilsoliton the rank-one pseudo-Iwasawa extension is the same Lie algebra for all the diagonal metrics, and it is isomorphic to the rank-one extension obtained using the Nikolayevsky derivation. In practice, once one knows that a nice nilpotent Lie algebra admits a \ref{cond:nil4}-type diagonal nilsoliton metric (e.g. by Theorem~\ref{thm:Xequalsb}), one may conclude that $\g\rtimes_N\Span{e_0}$ is Einstein with nonzero scalar curvature.

This rank-one extension obtained using the Nikolayevsky derivation is the standard extension used in the Riemannian setting.
\end{remark}

\begin{remark}
In the Riemannian setting, the clear relation between nilsoliton metrics and Einstein solvmanifolds suggest a strategy to build Einstein (almost)-K\"ahler metrics on solvable Lie algebra: first one looks for Riemannian nilsoliton metrics and then extends the Lie algebra to a solvable Riemannian Einstein. After that one searches for an invariant symplectic $2$-form $\omega$, and finally one checks if the endomorphism $J$ defined by $\omega(X,Y)=g(JX,Y)$ is a complex integrable stucture. This strategy was pursued efficiently in~\cite{Manero:phdThesis,FerFinMan:G2EinSolv}.

The similar relation between nilsoliton and Einstein metrics that we proved in the pseudo-Riemannian setting is the key point to extend this strategy to the indefinite case: our strategy, exposed in Section~\ref{sec:strategy}, includes the Riemannian one as a special case.
\end{remark}

\begin{remark}
We finally remark that our correspondence between Einstein solvmanifolds and nilsoliton metrics do not cover the whole space of pseudo-Riemannian Einstein solvmanifolds: there are examples that are not pseudo-Iwasawa extension or that are not related to nilsolitons. Nevertheless, the theory and technique developed up to now are very effective in construct such examples.
\end{remark}

\section{A Strategy and New Examples}\label{sec:strategy}

We remark that to build Einstein metrics with $s\neq0$, it is necessary to look for solvable Lie algebras that are not nilpotent, and a practical way to build them is through nilsoliton metrics. Using the results exposed in the previous section, we can write the following recipe to construct examples of Einstein non-Ricci-flat para-K\"ahler and pseudo-K\"ahler metrics on solvable non-nilpotent Lie algebras:
\begin{enumerate}[label={$\widehat{\mathrm{\arabic*}}$}]
\item\label{strategy:1} Find a nilsoliton metric $(\g,g)$ of type~\ref{cond:nil4}.
\item\label{strategy:2} Extend it to a Einstein pseudo-Riemannian solvable Lie algebra $(\tilde{\g},\tilde{g})$ using Theorem~\ref{thm:constructionstandardextension}, i.e. consider a pseudo-Iwasawa extension. Necessarily this extension is Einstein with nonzero scalar curvature.
\item\label{strategy:3} Look for symplectic $2$-form $\omega$, such that $\nabla \omega=0$.
\item\label{strategy:4} \textbf{(a)} Look for a integrable complex structure $J$ such that $g(J\cdot,\cdot)=\omega(\cdot,\cdot)$ and $\nabla J=0$. We got a \emph{pseudo-K\"ahler Einstein} solvable Lie algebra with nonzero scalar curvature.\\
Or\\
\textbf{(b)} Look for a integrable para-complex structure $K$ such that $g(K\cdot,\cdot)=\omega(\cdot,\cdot)$ and $\nabla K=0$. We got a \emph{para-K\"ahler Einstein} solvable Lie algebra with nonzero scalar curvature.
\end{enumerate}

In our strategy, examples of \ref{cond:nil4}-type nilsolitons required in point~\ref{strategy:1} can be constructed using nice algebras and diagonal nilsoliton metrics, i.e. by using Theorem~\ref{thm:Xequalsb}. Point~\ref{strategy:2} is fulfilled by applying the Theorem~\ref{thm:constructionstandardextension} (or by applying Proposition~\ref{prop:NicePseudoIwasawa} for diagonal nice nilsoliton). The points~\ref{strategy:3} and~\ref{strategy:4} on Lie algebras are standard problems. Once we know the pseudo-Iwasawa extension from point~\ref{strategy:2}, we can compute the space of closed $2$-forms and check if there is a nondegenerate one. Moreover, since the Einstein metric is also given by the previous point, we can easily compute the Levi-Civita connection and find a suitable symplectic parallel $2$-form. Finally, to deal with point~\ref{strategy:4}, we define an endomorphism $E$ such that $\omega(X,Y)=g(EX,Y)$, and we force it to be an integrable (para-) complex structure to achieve the final result, i.e. we impose that $N_E=0$ (or equivalently $\nabla E =0$) and that $E^2=\pm\id$.

In the end, if we restrict to look for nice diagonal \ref{cond:nil4}-type nilsoliton metrics most of this strategy becomes a linear problem, and hence is quite easy to deal with it.

Note that by Remark~\ref{rem:RankOneNiceExtension}, the second point~\ref{strategy:2} can be simplified by considering only rank-one extension of nilsoliton metric. As an example, we apply this strategy to some nilsoliton metrics obtained in~\cite{ContiRossi:NiceNilsolitons}: to this end, we recall in the following Table~\ref{Table2:NilLowDim} the low dimensional nilpotent Lie algebra and their Nikolayevsky derivation $N$ (note that up to dimension $5$ all the nilpotent Lie algebras are nice). The names refer to the notation used in~\cite{ContiRossi:Construction}, namely there are first the dimensions of the lower central series and a counter after the column.

\begin{table}[thp]
{\small\centering
\begin{tabular}{>{\ttfamily}r L L}
\toprule
\textnormal{Name} & \g & N \\
\midrule
1:1&0&(1)\\[2pt]
2:1&0,0&(1,1)\\[2pt]
31:1&0,0,e^{12}&2/3(1,1,2)\\
3:1&0,0,0&(1,1,1)\\[2pt]
421:1&0,0,e^{12},e^{13}&1/3(1,2,3,4)\\
41:1&0,0,0,e^{12}&1/3(2,2,3,4)\\
4:1&0,0,0,0&(1,1,1,1)\\[2pt]
5321:1&0,0,e^{12},e^{13},e^{14}&1/12(2,9,11,13,15)\\
5321:2&0,0,e^{12},e^{13},e^{14}+e^{23}&3/11(1,2,3,4,5)\\
532:1&0,0,e^{12},e^{13},e^{23}&5/12(1,1,2,3,3)\\
521:1&0,0,0,0,e^{12},e^{14}&1/3(1,2,3,3,4)\\
521:2&0,0,0,e^{12},e^{24}+e^{13}&1/7(4,3,6,7,10)\\
52:1&0,0,0,e^{12},e^{13}&1/4(2,3,3,5,5)\\
51:1&0,0,0,0,e^{12}&1/3(2,2,3,3,4)\\
51:2&0,0,0,0,e^{12}+e^{34}&3/4(1,1,1,1,2)\\
5:1&0,0,0,0,0&(1,1,1,1,1)\\
\bottomrule
\end{tabular}
\caption{Nice nilpotent Lie algebras of dimension $\leq 5$ and their diagonal Nikolayevsky derivation\label{Table2:NilLowDim}}
}
\end{table}

\begin{example}\label{ex:HeisenbergPseudoExtensionPK}
Consider the $3$-dimensional Heisenberg Lie algebra $\lie{h}$ with structure equations give by:
\[(0,0,e^{12}).\]
By~\cite{ContiRossi:NiceNilsolitons} we know it admits a diagonal nilsoliton metric $g$, and we can compute such a metric using Theorem~\ref{thm:Xequalsb}. The solution is not unique, and we have that
\[g=g_1e^1\otimes e^1 + g_2e^2\otimes e^2-\frac{2g_1g_2\lambda}{3}e^3\otimes e^3\]
is a nilsoliton metric satisfying $\Ric_g=\lambda(\id-N)$ where $N$ is the diagonal Nikoalyevsky derivation of $\lie{h}$, i.e.
\[N=\frac{2}{3}e^1\otimes e_1+\frac{2}{3}e^2\otimes e_2+\frac{4}{3}e^3\otimes e_3.\]
Thus the solvable Lie algebra $\tilde{\lie{g}}$ of dimension $4$ given by
\[\left(\frac{2}{3}e^{14},\frac{2}{3}e^{24},\frac{4}{3}e^{34}+e^{12},0\right)\]
admits an Einstein metric $\tilde{g}=g-\frac{\Tr N}{\lambda}e^4\otimes e^4$. We fix $\lambda =-\frac{1}{2}$, hence
\[\tilde{g}=g_1e^1\otimes e^1 + g_2e^2\otimes e^2+\frac{g_1g_2}{3}e^3\otimes e^3+\frac{16}{3}e^4\otimes e^4.\]
A simple linear computation show that a generic closed $2$-form can be write as
$ye^{12}+\frac{4y}{3}e^{34}+\Span\{e^{14},e^{24}\}$
and imposing $\dd\omega=0$, $\nabla \omega= 0$ and $\omega^2\neq0$ one gets that
\[\omega= y e^{12}+\frac{4 }{3}ye^{34},\qquad y\neq0.\]
At this point, we compute the endomorphis $E$ such that $g(EX,Y)=\omega(X,Y)$:
\[
E= -\frac{y}{g_1}e^2\otimes e_1 +\frac{y}{g_2} e^1\otimes e_2-\frac{4y}{g_1g_2}e^4\otimes e_3+\frac{y}{4} e^{3}\otimes e_4.
\]
Finally, we impose that $\nabla E=0$ (equivalently that $N_E=0$) and $E^2=\pm \id$. For $E=-\Id$ one gets $g_2=\frac{y^2}{g_1}$ while for $E=\Id$ one gets $g_2=-\frac{y^2}{g_1}$. In the end we have that
\[\left\{
\begin{aligned}
&\tilde{g}=g_1e^1\otimes e^1 +\frac{y^2}{g_1} e^2\otimes e^2+\frac{y^2}{3}e^3\otimes e^3+\frac{16}{3}e^4\otimes e^4,\\
&\omega= y e^{12}+\frac{4 }{3}ye^{34},\\
&J(e_1)=\frac{g_1}{y} e_2,\qquad J(e_2)= -\frac{y}{g_1} e_1,\qquad J(e_3)=\frac{y}{4} e_4,\qquad J(e_4)=-\frac{4}{y}e_3,
\end{aligned}
\right.\]
is a pseudo-k\"{a}hler Einstein structure; and
\[\left\{
\begin{aligned}
&\tilde{g}=g_1e^1\otimes e^1 -\frac{y^2}{g_1} e^2\otimes e^2-\frac{y^2}{3}e^3\otimes e^3+\frac{16}{3}e^4\otimes e^4,\\
&\omega=y e^{12}+\frac{4 }{3}ye^{34},\\
&K(e_1)=-\frac{g_1}{y}e_2,\qquad K(e_2)= -\frac{y}{g_1} e_1,\qquad K(e_3)=\frac{y}{4}e_4,\qquad K(e_4)=\frac{4}{y}e_3,
\end{aligned}
\right.\]
is a para-k\"{a}hler Einstein structure. In both cases $\Ric_{\tilde{g}}=-\frac{1}{2}\id$.
\end{example}

We would like to make another example, in dimension $6$.
\begin{example}\label{ex:51:2PseudoExtensionPK}
Consider the $5$-dimensional Lie algebra \texttt{51:2} with structure equation give by:
\[(0,0,0,0,e^{12}+e^{34}).\]
By \cite{ContiRossi:NiceNilsolitons} we know it admits a diagonal nilsoliton metric $g$, and we can compute such a metric using Theorem~\ref{thm:Xequalsb}, hence we consider the rank-one pseudo-Iwasawa solvable extension obtained using the Nikolayevsky derivation $N$, namely the Lie algebra $\tilde{\g}$
\[\left(\frac{3}{4}e^{16},\frac{3}{4}e^{26},\frac{3}{4}e^{36},\frac{3}{4}e^{46},e^{12}+e^{34}\frac{6}{4}e^{56},0\right).\]
Proceding as before, on $\tilde{\g}$ we find a pseudo-k\"{a}hler Einstein structure given by:
\[\left\{
\begin{aligned}
&\tilde{g}=g_1e^1\otimes e^1 +\frac{y^2}{g_1} e^2\otimes e^2+g_3e^3\otimes e^3 +\frac{y^2}{g_3} e^4\otimes e^4+\frac{y^2}{4}e^5\otimes e^5+9e^6\otimes e^6,\\
&\omega= y e^{12}+ ye^{34}+\frac{3 }{2}ye^{56},\\
&J(e_1)=\frac{g_1}{y} e_2,\qquad J(e_2)= -\frac{y}{g_1} e_1,\qquad J(e_3)=\frac{g_3}{y} e_4,\\
&J(e_4)= -\frac{y}{g_3} e_3,\qquad J(e_5)=\frac{y}{6} e_6,\qquad J(e_6)=-\frac{6}{y}e_5;
\end{aligned}
\right.\]
and a para-k\"{a}hler Einstein structure given by:
\[\left\{
\begin{aligned}
&\tilde{g}=g_1e^1\otimes e^1 -\frac{y^2}{g_1} e^2\otimes e^2+g_3e^3\otimes e^3 -\frac{y^2}{g_3} e^4\otimes e^4-\frac{y^2}{4}e^5\otimes e^5+9e^6\otimes e^6,\\
&\omega= y e^{12}+ ye^{34}+\frac{3 }{2}ye^{56},\\
&K(e_1)=-\frac{g_1}{y} e_2,\qquad K(e_2)= -\frac{y}{g_1} e_1,\qquad K(e_3)=-\frac{g_3}{y} e_4,\\
&K(e_4)= -\frac{y}{g_3} e_3,\qquad K(e_5)=\frac{y}{6} e_6,\qquad K(e_6)=\frac{6}{y}e_5.
\end{aligned}
\right.\]
In both cases $\Ric_{\tilde{\g}}=-\frac{1}{2}$.
\end{example}

\begin{remark}\label{rem:NotUnimodular}
Since we use nilsoliton of type~\ref{cond:nil4}, it follows from Corollary~\ref{cor:Hzero} that the pseudo-Iwasawa extension are not unimodular. On the other hand, Einstein nilpotent Lie algebras are unimodular, and most of the Lie algebras we produced in~\cite{ContiRossi:EinsteinNice} have rational constant; therefore the associated Lie group $G$ has a lattice $\Gamma$ and the left-invariant Einstein metric on $G$ induces an Einstein metric on the nilmanifold $\Gamma\backslash G$.
\end{remark}

A clear obstruction to the construction of pseudo-K\"ahler and para-K\"ahler metrics is the existence of a symplectic structure. For the rank-one pseudo-Iwasawa extension of Remark~\ref{rem:RankOneNiceExtension}, one can easily check if the target Lie algebra admits such a structure. This test produces the following result.
\begin{lemma}\label{lem:PseudoIwasawaSymplectic}
Let $\g$ be a nice nilpotent Lie algebra of dimension $\leq 5$ and let $g$ be a diagonal \ref{cond:nil4}-type nilsoltion. Then the rank-one pseudo-Iwasawa extension $(\tilde{\g},\tilde{g})$ admits a symplectic structure if, and only if, it is the pseudo-Iwasawa extension of one the following Lie algebras:
\begin{center}
\begin{tabular}{>{\ttfamily}rL}
1:1&(0)\\
31:1&(0,0,e^{12})\\
5321:2&(0,0,e^{12},e^{13},e^{14}+e^{23})\\
521:2&(0,0,0,e^{12},e^{24}+e^{13})\\
51:2&(0,0,0,0,e^{12}+e^{34})
\end{tabular}
\end{center}
\end{lemma}

\begin{proof}
We can exclude all the Lie algebra $\g$ of even dimension, since their rank-one extension can not admit any symplectic structure.

For the remaining case, we know by \cite[Theorem~2.1]{ContiRossi:NiceNilsolitons} that all those examples admit a nilsoliton of type~\ref{cond:nil4}, and Theorem~\ref{thm:Xequalsb} tells us that the derivation $D$ appearing in~\eqref{eq:nilsoliton} is a multiple of the Nikolayevsky derivation $N$. Hence, reasoning as in Remark~\ref{rem:RankOneNiceExtension}, the structure equations of the pseudo-Iwasawa extension $\tilde{\g}$ are easily computed using Proposition~\ref{prop:NicePseudoIwasawa} and Table~\ref{Table2:NilLowDim}: for example, the pseudo-Iwasawa extension of
\texttt{521:2} is the solvable Lie algebra $\tilde{\g}=\g\rtimes_N\Span{e_6}$, where $N$ is the diagonal Nikolayevsky derivation and $[e_6,e_i]=N(e_i)$. Explicitly, the structure equations are
\[\left(\frac{4}{7}e^{16},\frac{3}{7}e^{26},\frac{6}{7}e^{36},e^{46}+e^{12},\frac{10}{7}e^{56}+e^{24}+e^{13},0\right).\]

Finally, a direct computation of the space of $2$-forms shows that the Lie algebras not listed in the statement do not have any nondegenerate closed $2$-form. For the other Lie algebraa, we explicitly give the structure equations of the extension and a symplectic $2$-form $\omega$ in the Table~\ref{table:FormeSimplettiche}.
\end{proof}

{\setlength{\tabcolsep}{4pt}
\begin{table}[thp]
{\small\centering
\begin{tabular}{>{\ttfamily}r C L}
\toprule
\textnormal{Name of $\g$} & \textnormal{Rank-one extension }\tilde{\g} & \textnormal{Symplectic form } \omega\\
\midrule
1:1&e^{12},0&e^{12}\\[5pt]
31:1&\dfrac{2}{3}e^{14},\dfrac{2}{3}e^{24},\dfrac{4}{3}e^{34}+e^{12},0&e^{12}+\dfrac{4}{3}e^{34}\\[5pt]
5321:2&\dfrac{3}{11}e^{16},\dfrac{6}{11}e^{26},\dfrac{9}{11}e^{36}+e^{12},\dfrac{12}{11}e^{46}+e^{13},\dfrac{15}{11}e^{56}+e^{14}+e^{23},0& e^{14}+e^{23}+\dfrac{15}{11}e^{56}\\[5pt]
521:2&\dfrac{4}{7}e^{16},\dfrac{3}{7}e^{26},\dfrac{6}{7}e^{36},e^{46}+e^{12},\dfrac{10}{7}e^{56}+e^{24}+e^{13},0&
e^{13}+e^{24}+\dfrac{10}{7}e^{56}\\[5pt]
51:2&\dfrac{3}{4}e^{16},\dfrac{3}{4}e^{26},\dfrac{3}{4}e^{36},\dfrac{3}{4}e^{46},\dfrac{6}{4}e^{56}+e^{12}+e^{34},0&
e^{12}+e^{34}+\dfrac{3 }{2}e^{56}\\
\bottomrule
\end{tabular}
\caption{\label{table:FormeSimplettiche} Symplectic rank-one pseudo-Iwasawa extensions of diagonal nice nilsoltions of dimension $\leq6$}
}
\end{table}
}

We conclude this section classifying the rank-one pseudo-Iwasawa extension of nice nilsolton diagonal metric admitting a pseudo-K\"ahler or para-K\"ahler Einstein metric with nonzero scalar curvature.
\begin{theorem}\label{thm:NiceDiagonalPseudoIwasawaLowDim}
Let $\g$ be a nice nilpotent Lie algebra of dimension $\leq 5$ with a diagonal nilsoliton $g$ and let $(\tilde{\g},\tilde{g})$ be its rank-one pseudo-Iwasawa extension. Then the following are equivalent:
\begin{enumerate}
\item $(\tilde{\g},g)$ admits an Einstein pseudo-K\"ahler structure with nonzero scalar curvature;
\item $(\tilde{\g},g)$ admits an Einstein para-K\"ahler structure with nonzero scalar curvature;
\item depending on the dimension $\dim\tilde{\g}$, the pair $(\g,\tilde{\g})$ are isomorphic respectively to
\begin{itemize}
\item $(0)\quad\text{and}\quad (e^{12},0)$; or
\item $(0,0,e^{12})\quad\text{and}\quad\left(\dfrac{2}{3}e^{14},\dfrac{2}{3}e^{24},\dfrac{4}{3}e^{34}+e^{12},0\right)$; or
\item $(0,0,0,0,e^{12}+e^{34})\quad\text{and}\quad\left(\dfrac{3}{4}e^{16},\dfrac{3}{4}e^{26},\dfrac{3}{4}e^{36},\dfrac{3}{4}e^{46},\dfrac{6}{4}e^{56}+e^{12}+e^{34},0\right)$.
\end{itemize}
\end{enumerate}
\end{theorem}
\begin{proof}
The Examples~\ref{ex:HeisenbergPseudoExtensionPK}, \ref{ex:51:2PseudoExtensionPK} and~\ref{ex:AffPK} show that $3$ implies both $1$ and $2$.
For the converse, note that Lemma~\ref{lem:PseudoIwasawaSymplectic} reduce the possibility to only $5$ possible pseudo-Iwasawa extension, which are listed explicitly in Table~\ref{table:FormeSimplettiche}. Now observe that by Proposition~\ref{prop:NicePseudoIwasawa} and the definition of rank-one pseudo-Iwasawa extension the Einstein metric $\tilde{g}$ is diagonal on $\tilde{\g}$ with respect to the basis given in Table~\ref{table:FormeSimplettiche}.

We compute the Levi-Civita connection of the Einstein metric $\tilde{g}=g_ie^i\otimes e ^i$, and taking a closed $2$-form $\omega=\sum y_{i,j}e^{ij}$ we look for a suitable symplectic parallel $2$-form.

Consider the extension of \texttt{5321:2}. Then we have
\begin{align*}
\nabla_{e_1}\omega(e_2,e_4)&=-\omega(\nabla_{e_1}e_2,e_4)-\omega(e_2,\nabla_{e_1}e_4)\\
&=-\omega(-\frac{1}{2}e_3,e_4)-\omega(e_2,\frac{g_4}{2g_3}e_3-\frac{1}{2}e_5)=-\frac{y_{2,3}g_4}{2g_3}
\end{align*}
because any closed $2$-form satisfies $\omega(e_3,e_4)=\omega(e_2,e_5)=0$ and $\omega(e_2,e_3)=y_{2,3}\neq 0$.

Similarily, for the extension of \texttt{521:2}, we got
\begin{align*}
\nabla_{e_1}\omega(e_1,e_2)&=-\omega(\nabla_{e_1}e_1,e_2)-\omega(e_1,\nabla_{e_1}e_2)\\
&=-\omega(\frac{4g_1}{7g_6}e_6,e_4)-\omega(e_1,-\frac{1}{2}e_4)=+\frac{4g_1y_{2,6}}{7g_6}
\end{align*}
because any closed $2$-form satisfies $\omega(e_1,e_4)=0$ and $\omega(e_2,e_6)=y_{2,6}\neq 0$.

We conclude that neither the pseudo-Iwasawa Einstein extension of \texttt{5321:2} nor the pseudo-Iwasawa Einstein extension of \texttt{521:2} can admit a parallel symplectic structure, hence they have no pseudo-K\"ahler and no para-K\"ahler structure compatible with the Einstein metric.
\end{proof}

\section{More Examples, Ideas and Remarks}\label{sec:examples}

In this Section we will list some possible implementations of our strategy as well as different useful approach to construct special Einstein metrics on solvable Lie algebra with nonzero scalar curvature.

\medskip
\textbf{1.}
Instead of working with diagonal nice nilsoliton metric, one can look for non diagonal nilsoliton metrics, and then apply Theorem~\ref{thm:constructionstandardextension} to construct an Einstein solvable Lie algebra. In this case the challenge part is to find a nilsoliton metric. With this approach, we were able to find a structure on a pseudo-Iwasawa extension of \texttt{521:2}, which was excluded from Theorem~\ref{thm:NiceDiagonalPseudoIwasawaLowDim}.

\begin{example}\label{ex:521:2ParaStorta}
The metric $g=g_1e^1\odot e^3-g_1e^2\odot e^4-\frac{g_1^2}{4}e^{5}\otimes e^{5}$ is a \ref{cond:nil4}-type nilsoliton on the Lie algebra \texttt{521:2}, with a semisimple diagonal derivation $D$ (thus $D$ is multiple of the Nikolayevsky derivation). Hence by Theorem~\ref{thm:constructionstandardextension} on the pseudo-Iwasawa solvable extension
\[\left(\dfrac{4}{7}e^{16},\dfrac{3}{7}e^{26},\dfrac{6}{7}e^{36},e^{46}+e^{12},\dfrac{10}{7}e^{56}+e^{24}+e^{13},0\right)\]
the metric $\tilde{g}=g_1e^1\odot e^3-g_1e^2\odot e^4-\frac{g_1^2}{4}e^{5}\otimes e^{5}+\frac{400}{49}e^{6}\otimes e^{6}$ is Einstein ($\Ric_{\tilde{\g}}=-\frac{1}{2}$). Moreover, a straightforward computation prove that the triple
\[\left\{
\begin{aligned}
&\tilde{g},\qquad\omega= g_1e^{13}+g_1e^{24}+\dfrac{10}{7}g_1e^{56},\\
&K(e_1)=e_1,\qquad K(e_2)= -e_2,\qquad K(e_3)=- e_3,\qquad K(e_4)= -e_4,\\
&K(e_5)=-\dfrac{10}{7}g_1 e_6,\qquad K(e_6)=\dfrac{10}{7}g_1e_5,
\end{aligned}
\right.\]
defines a para-K\"ahler structure compatible with the metric $\tilde{g}$.
\end{example}

\medskip
\textbf{2.} Given a \ref{cond:nil4}-type nilsoliton metric, another possibility is to compute a higher rank pseudo-Iwasawa extension: by Theorem~\ref{thm:constructionstandardextension} this can be achieved computing a set $\lie{a}$ of commuting self-adjoint derivation containing $D$, such that the scalar product induced by the trace is nondegenerate.

\begin{example}\label{ex:AbelianRank2ExtensionPseudoK}
Consider the abelian Lie algebra of dimension $2$ with basis $\{e_1,e_2\}$ together with the Lorentzian nilsoltion metric $g=\frac{1}{4} (-e^1\otimes e^1+e^2\otimes e^2)$, and consider the self-adjoint derivations $N=e^1\otimes e_1+e^2\otimes e_2$ and $D=e^2 \otimes e_1- e^1\otimes e_2$ (note that $N$ is the Nikolayevsky derivation). Those derivations commute and $\langle,\rangle_{\Tr}$ is nondegenerate on $\Span\{N,D\}$, so we have a $4$-dimensional solvable pseudo-Iwasawa Einstein extension $\tilde{\g}$ with structure equation given by:
\[(e^{13}+e^{24} ,e^{23} - e^{14},0,0)\]
and Einstein metric $\tilde{g}$ given by
\[\tilde{g}=\frac{1}{4} (-e^1\otimes e^1+e^2\otimes e^2)+4(e^3\otimes e^3-e^4\otimes e^4)\]
such that $\Ric_{\tilde{g}}=-\frac{1}{2}\id$. Finally, it is a matter of computation to verify that $\omega=-e^{14}+e^{23}$ is a parallel symplectic structure and the endomorphism
\[J= -4e^4\otimes e_1 -4e^3\otimes e_2 +\frac{1}{4}e^2\otimes e_3+\frac{1}{4}e^1\otimes e_4\]
is an integrable complex structure such that $g(J\cdot,\cdot)=\omega(\cdot,\cdot)$. Hence we have an Einstein pseudo-K\"ahler structure.
\end{example}

Note that the metric appearing in Example~\ref{ex:AbelianRank2ExtensionPseudoK} has neutral signature but it is possible to modify the metric to obtain a Riemannian Einstein K\"ahler metric.

\begin{example}
As in the previous example, let $\g$ be the $2$-dimensional abelian Lie algebra $\{e_1,e_2\}$ together with the Lorentzian nilsoltion metric $g=\frac{1}{4} (e^1\otimes e^1-e^2\otimes e^2)$, and consider the self-adjoint derivations $N$ and $D$ as in Example~\ref{ex:AbelianRank2ExtensionPseudoK}. We have a $4$-dimensional solvable pseudo-Iwasawa Einstein extension $\tilde{\g}$ as before.
The Einstein metric $\tilde{g}$ is given by
\[g=\frac{1}{4} (e^1\otimes e^1-e^2\otimes e^2)+4(e^3\otimes e^3-e^4\otimes e^4)\]
such that $\Ric_{\tilde{g}}=-\frac{1}{2}\id$. Again, the $2$-form $\omega$ is a parallel symplectic structure and the endomorphism
\[K= 4e^4\otimes e_1 +4e^3\otimes e_2 +\frac{1}{4}e^2\otimes e_3+\frac{1}{4}e^1\otimes e_4\]
is an integrable para-complex structure such that $g(K\cdot,\cdot)=\omega(\cdot,\cdot)$, hence we have an Einstein para-K\"ahler structure.
\end{example}

As mentioned before, this procedure may lead also to Rimannian Einstein K\"ahler structure. The following example will show a concrete case.

\begin{example}\label{ex:DoubleExtension421:1}
Let $\g$ be the nilpotent Lie algebra \texttt{421:1}. From \cite{ContiRossi:NiceNilsolitons} we know it admits a diagonal nilsoliton metric of \ref{cond:nil4}-type, hence the rank-one pseudo-Iwasawa extension has an Einstein metric. Note that the derivations
\begin{gather*}
N=\frac{1}{3}(e^1\otimes e_1+2e^2\otimes e_2+ 3e^3\otimes e_3+4e^4\otimes e_4),\\
D=\frac{1}{3}(3e^1\otimes e_1-4e^2\otimes e_2-e^3\otimes e_3+24e^4\otimes e_4)
\end{gather*}
commute and are both self-adjoint with respect to the diagonal nilsoliton metric (in fact, $N$ is the Nikolayevsky derivation). Since $\langle,\rangle_{\Tr}$ is nondegenerate in the space $\Span\{N,D\}=\lie{a}$, by Theorem~\ref{thm:constructionstandardextension} we can conclude that the rank-two pseudo-Iwasawa extension $\tilde{\g}=\g\rtimes\lie{a}$ has an Einstein metric. Explicitly $\tilde{\g}$ is given by
\[\left(\frac{1}{3}e^{15}+e^{16},\frac{2}{3}e^{25}-\frac{4}{3}e^{26},e^{35}-\frac{1}{3}e^{36}+e^{12},\frac{4}{3}e^{45}+\frac{2}{3}e^{46}+e^{13},0,0\right)\]
and we can choose the Einstein metric $\tilde{g}$ as
\[\frac{3}{g_1}e^1\otimes e^1+3g_1e^2\otimes e^2+3 g_1e^3\otimes e^3+ 3e^4\otimes e^4+\frac{20}{3}e^5\otimes e^5+\frac{20}{3}e^6\otimes e^6\]
with $\Ric_{\tilde{g}}=-\frac{1}{2}\id$ (the choice of $\tilde{g}$ is not unique, depends on diagonal parameters of the nilsoliton metrics, but doing so we will simplify further computation).
Finally, proceeding as in Example~\ref{ex:HeisenbergPseudoExtensionPK} we find that the triple
\[\left\{
\begin{aligned}
&\tilde{g},\qquad\omega=  3e^{13}+2 g_1e^{25} -4 g_1e^{26}+ 4e^{45}+ 2e^{46},\\
&J(e_1)=\frac{1}{g_1}e_1,\qquad J(e_2)= \dfrac{3g_1}{10}e_5-\dfrac{3g_1}{5}e_6,\qquad J(e_3)=- g_1e_1,\\
&J(e_4)= \dfrac{3}{5}e_5+\dfrac{3}{10}e_6,\qquad
J(e_5)=-\dfrac{2}{3g_1} e_2-\dfrac{4}{3}e_4,\qquad J(e_6)=\dfrac{4}{3g_1}e_2-\dfrac{2}{3}e_4,
\end{aligned}
\right.\]
defines a pseudo-K\"ahler structure compatible with the metric $\tilde{g}$.

Similarly, again on the solvable Lie algebra $\tilde{\g}=\g\rtimes\lie{a}$, we can chose the triple
\[\left\{
\begin{aligned}
&\tilde{g}=-\frac{3}{g_1}e^1\otimes e^1-3g_1e^2\otimes e^2+3 g_1e^3\otimes e^3- 3e^4\otimes e^4+\frac{20}{3}e^5\otimes e^5+\frac{20}{3}e^6\otimes e^6,\\
&\omega=  3e^{13}-2 g_1e^{25} +4 g_1e^{26}+ 4e^{45}+ 2e^{46},\\
&K(e_1)=\frac{1}{g_1}e_1,\qquad K(e_2)=-\dfrac{3g_1}{10}e_5+\dfrac{3g_1}{5}e_6,\qquad K(e_3)=g_1e_1,\\
&K(e_4)= \dfrac{3}{5}e_5+\dfrac{3}{10}e_6,\qquad
K(e_5)=-\dfrac{2}{3g_1} e_2+\dfrac{4}{3}e_4,\qquad K(e_6)=\dfrac{4}{3g_1}e_2+\dfrac{2}{3}e_4,
\end{aligned}
\right.\]
which defines an Einstein para-K\"ahler structure with $\Ric_{\tilde{g}}=\frac{1}{2}\id$
\end{example}

\begin{remark}\label{rem:ManeroMisprint}
We observe that in the example above the parameters of the Einstein pseudo-K\"ahler metric $\tilde{g}$ can be chosen to be all positive, i.e. for $g_1>0$ we have a K\"ahler Einstein complex structure. This solvable Lie algebra $\tilde{\g}$ of Example~\ref{ex:DoubleExtension421:1} was studied in~\cite[Theorem~3.2.2]{Manero:phdThesis} (see also~\cite[Theorem~3.1]{FerFinMan:G2EinSolv}) where it is stated that $\tilde{\g}$ admits only a K\"ahler Einstein almost complex structure which is not integrable. This is in contrast with our previous example: indeed, we find a misprint in the definition of the endomorphism $J$ in the proof of~\cite[Theorem~3.2.2]{Manero:phdThesis}, that might have led the authors to an incorrect conclusion.
\end{remark}

\medskip
\textbf{3.} We noted in Remark~\ref{rem:NotUnimodular} that the pseudo-Iwasawa extension obatained using \ref{cond:nil4}-type nilsoliton are necessarily non unimodular, and hence by \cite[Lemma~6.2]{Milnor:curvatures} they do not have a discrete subgroup with compact quotient.

To have compact Einstein pseudo-K\"ahler or para-K\"ahler manifolds with $s\neq0$ arising as quotient of solvable Lie groups by a discrete lattice, we need to look for unimodular Lie algebras. It may be possible to obtain Einstein unimodular solvable Lie algebras with nonzero scalar curvature using extension of \ref{cond:nil3}-type nilsoliton (see \cite[Corollary~4.12]{ContiRossi:IndefiniteNilsolitons}). Once the solvable extension is fixed, we may apply points~\ref{strategy:3} and~\ref{strategy:4} as before.

Note that \ref{cond:nil3}-type nilsolitons are Einstein metrics with $s\neq0$ on nilpotent Lie algebras, and they appear for dimensions $\geq7$ (\cite{ContiRossi:EinsteinNilpotent,FerFreSan20:7dimEinsteinNilpotent}). In this case the computations may be more challenging.

Moreover, we observe that by Proposition~\ref{pr:ContiRossiInvariantParaK}, this way can not produce any Einstein para-K\"ahler metric with nonzero scalar curvature.

At time of writing, we are not aware of a similar obstruction for Einstein pseudo-K\"ahler metrics on unimodular Lie algebras, nor if they exist: if it is so, the above discussion can lead explicitly to such an example.

\medskip
\textbf{4.} We see that Examples~\ref{ex:HeisenbergPseudoExtensionPK} and~\ref{ex:51:2PseudoExtensionPK} in low dimensions all belong to the family of rank-one pseudo-Iwasawa extension of the generalized Heisenberg Lie algebra. In fact, that construction can be generalized, as we will show.

\begin{example}\label{ex:GeneralizedHeisenberg}
For any $1\leq n\in\N$, let us consider the generalized Heisenberg Lie algebra, that is the nilpotent $(2n+1)$-dimensional Lie algebra $\lie{h}_{2n+1}$ with structure equations given by:
\[\left(0,\dots,0, \sum_{i=1}^{n}e^{2i-1,2i}\right)\]
equivalently, $[e_{2i-1},e_{2i}]=-e_{2n+1}$ for $i=1,\dots,n$.
This is a nice Lie algebra, and its Nikolayevsky derivation is
\[N=\sum_{i=1}^{n} \frac{n+1}{n+2}(e^{2i-1}\otimes e_{2i-1}+e^{2i-1}\otimes e_{2i-1})+2\frac{n+1}{n+2}e^{2n+1}\otimes e_{2n+1}.\]
It admits diagonal nilsoliton pseudo-Riemannian metric of type~\ref{cond:nil4}. In particular, we chose the following diagonal family of metrics
\[g_\eps=\sum_{i=1}^n\left(g_{2i-1}e^{2i-1}\otimes e^{2i-1}-\eps\frac{\alpha^2}{g_{2i-1}}e^{2i}\otimes e^{2i}\right)-\eps\frac{\alpha^2}{n+1}e^{2n+1}\otimes e^{2n+1},\]
where $\eps\in\{+1,-1\}$ and $\alpha,\ g_{2i-1}\neq 0$. An easy computation shows that $\Ric_\eps = -\frac{1}{2}\Id+\frac{1}{2} N$.

Consider now the rank-one pseudo-Iwasawa extension $\tilde{\g}=\g\rtimes\Span\{e_{2n+2}\}$ where $[e_{2n+2},e_j]=N e_j$. Namely the structure equations of $\tilde{\g}$ are given by
\[\left( \frac{n+1}{n+2}e^{1,2n+2},\dots,\frac{n+1}{n+2}e^{2n,2n+2}, 2\frac{n+1}{n+2}e^{2n+1,2n+2}+\sum_{i=1}^{n}e^{2i-1,2i},0\right),\]
and by Theorem~\ref{thm:constructionstandardextension} the metric $\tilde{g}_\eps=g_{\eps}+4\frac{(n+1)^2}{(n+2)}{e^{2n+2}\otimes e^{2n+1}}$ is an Einstein metric satisfying $\Ric_{\tilde{g}_\eps}=-\frac{1}{2}\id$. We note that $\tilde{\g}$ is a symplectic nonunimodular solvable Lie algebra of dimension $2n+2$, and we consider the nondegenerate closed $2$-form $\omega$ given by
\[\omega=\alpha \sum_{i=1}^{n}e^{2i-1,2i}+\alpha\frac{(n+1)^2}{n+2}(e^{2n+1,2n+2}).\]
Finally, we consider the endomorphism $J_{\eps}$ given by:
\begin{multline*}
J_{\eps}=\sum_{i=1}^n\left(-\eps\frac{g_{2i-1}}{\alpha}e^{2i-1}\otimes e_{2i}-\frac{\alpha}{g_{2i-1}}e^{2i}\otimes e_{2i-1}\right)\\
+\frac{\alpha}{2(n+1)}e^{2n+1}\otimes e_{2n+2}+\eps\frac{2(n+2)}{\alpha}e^{2n+2}\otimes e_{2n+1}.
\end{multline*}
It is easy to see that $J_{\eps}^2=\eps \Id$. Denoting the Levi-Civita connection of $g_{\eps}$ by $\nabla_{\eps}$, a straightforward computation shows that
\begin{gather*}
\nabla_\eps \omega=\nabla_\eps J_\eps=\nabla_\eps \tilde{g}_\eps=0,\quad N_{J_\eps}=0\\
\dd\omega=0,\quad\tilde{g}_{\eps}(J_\eps X,Y)=\omega(X,Y)\qquad X,Y\in\tilde{\g},
\end{gather*}
hence for $\eps=+1$, the triple $(g_{+1},\omega,J_{+1})$ defines an Einstein para-K\"ahler structure, and for $\eps=-1$ the triple $(g_{-1},\omega,J_{-1})$ defines an Einstein para-K\"ahler structure.

Note that the metric $\tilde{g}_{+1}$ has always neutral signature, while the Einstein metric $\tilde{g}_{-1}$ can have different signature, depending on the signs of the $n$ parameters $g_{2i-1}$. In particular the family of Einstein metric $\tilde{g}_{-1}$ can have any signature $\big(2(1+k),2(n-k)\big)$ for $k=0,\dots,n$.
\end{example}

The last example leads us to the
\begin{theorem}\label{thm:GenerHeisemberg}
On the rank-one pseudo-Iwasawa extension $\tilde{\g}$ of the generalized Heisenberg Lie algebra $\lie{h}_{2n+1}$, there exist an Einstein pseudo-K\"ahler structure and an Einstein para-K\"ahler structure. In particular there exists a Riemannian Einstein K\"ahler structure.
\end{theorem}
Note that the previous examples of Einstein K\"ahler metrics on rank-one extensions founded in \cite{Manero:phdThesis,FerFinMan:G2EinSolv} belong to the family described in Theorem~\ref{thm:GenerHeisemberg}.

\medskip
\textbf{5.} We can modify our strategy to build examples of other Einstein special metrics. For example, a similar quest for Ricci-flat pseudo-K\"ahler or para-K\"ahler metrics may be addressed with the same strategy explained in Section~\ref{sec:strategy}: by changing slightly the points~\ref{strategy:1} and~\ref{strategy:2} one may use the correspondence between nilsolitons and Ricci-flat extensions developed in~\cite{ContiRossi:IndefiniteNilsolitons}. Indeed, we have noted a large amount of Ricci-flat pseudo-Riemannian metric in literature, so in~\cite{ContiDelBarcoRossi:DiagramInvolutions} we wonder if every nilpotent Lie algebra admits a Ricci-flat metric. Note  that there exist nilpotent Lie algebras not admitting any flat metric  (\cite{AubMed03:PseudoRiemannianFlat}), hence the same question for flat metrics is false.

However, the abundance of Ricci-flat pseudo-Riemannian metrics and the fact that pseudo-K\"ahler and para-K\"ahler Ricci-flat metrics appear also on nilpotent Lie algebras, suggest that the nilsolitons may not be the best approach and we believe that there may be different ways to deal with this specific problem (indeed, the real challenge is to build Einstein pseudo-Riemannian metrics with nonzero scalar curvature).

Another class of interesting metrics are the nearly pseudo-K\"ahler and  nearly para-K\"ahler metrics (see e.g. \cite{SchSch-Hen:NearlyPseudoPara}), i.e. a triple $(g,J_\eps,\omega)$ where $g$ is a pseudo-Riemannian metric, $J_\eps$ is an almost (para-) complex and satisfy $J^2_{\eps}=\eps \Id$ for $\eps=\pm1$, $\omega$ is a $2$-form such that $g(J_\eps X, Y)=\omega(X,Y)$, and the following condition holds
\[(\nabla_X J_\eps)X=0,\]
where $\nabla$ is the Levi-Civita connection of $g$.
Those metrics are not symplectic nor (para-) complex, unless they are pseudo-K\"ahler or para-K\"ahler. We would like to point out that our strategy might be modified to construct nearly pseudo-K\"ahler (or nearly para-K\"ahler) Einstein metric with $s\neq0$ on solvable Lie algebras. In this case, one keeps points~\ref{strategy:1} and~\ref{strategy:2}, that produce the Einstein metric and hence the Levi-Civita connection, and modify the other points to achieve the construction of a nearly pseudo-K\"ahler (or nearly para-K\"ahler) metrics.

\medskip
\textbf{6.} The strategy proposed in Section~\ref{sec:strategy} can be weakened to search for almost complex (resp. almost para-complex) non integrable structure. In this case the symplectic structure and the endomorphis $J$ (resp. $K$) do not need to be parallel with respect to the Levi-Civita connection. The strategy to build such examples may follows points~\ref{strategy:1} and~\ref{strategy:2} to construct a suitable Einstein solvable Lie algebra, and then one modifies points \ref{strategy:3} and \ref{strategy:4}, by allowing $\nabla\omega\neq0$ and $N_{E}\neq0$.

We remark that almost pseudo-K\"ahler or para-K\"ahler Einstein structures with $s\neq0$ may appear also on nilpotent Lie algebras, because the restrictions of Lemma~\ref{lem:FPSInvariantPseudoNilpotent} and of Proposition~\ref{pr:ContiRossiInvariantParaK} do not apply. To our knowledge, the existence of invariant almost pseudo-K\"ahler or para-K\"ahler Einstein metrics with nonzero scalar curvature on nilpotent Lie algebras is an open question, that we feel may be out of scope for this article.

%
%\bibliographystyle{plainurl}
%\bibliography{PseudoPara_Einstein_SolvBIB.bib}

\small\noindent Dipartimento di Matematica e Applicazioni, Universit\`a di Milano Bicocca, via Cozzi 55, 20125 Milano, Italy.\\
\texttt{federico.rossi@unimib.it}

\newenvironment{dedication}
  {\clearpage
   \thispagestyle{empty}
   \vspace*{\stretch{1}}
   \itshape
   \raggedleft
  }
  {\par
   \vspace{\stretch{3}}
   \clearpage
  }

\begin{dedication}
{\Large \cursive{Dedicato al mio Babbo}}

\end{dedication}

\end{document}